\documentclass[reqno,oneside,12pt]{amsart}

\usepackage[english]{babel}
\usepackage[utf8]{inputenc}
\usepackage{amsfonts,amsmath,amsthm}
\usepackage{stmaryrd}
\usepackage{thmtools,thm-restate}
\usepackage[T1]{fontenc}
\usepackage{enumitem}
\usepackage{varioref}
\usepackage[all]{xy}
\usepackage{units}
\usepackage{hyperref}
\usepackage{graphicx}
\usepackage{amssymb}
\usepackage{mathabx}
\usepackage{times,mathptm, mathtools}
\usepackage{etoolbox}
\usepackage{tikz-cd}
\usepackage{mathrsfs}
\usepackage{mathdots}
\usepackage[left=2.3cm, right=2.3cm]{geometry}

\docsvlist{A,B,C,D,E,F,G,H,I,J,K,L,M,N,O,P,Q,R,S,T,U,V,W,X,Y,Z}

\docsvlist{A,B,C,D,E,F,G,H,I,J,K,L,M,N,O,P,Q,R,S,T,U,V,W,X,Y,Z}

\docsvlist{A,B,C,D,E,F,G,H,I,J,K,L,M,N,O,P,Q,R,S,T,U,V,W,X,Y,Z}

\docsvlist{A,B,C,D,E,F,G,H,I,J,K,L,M,N,O,P,Q,R,S,T,U,V,W,X,Y,Z}

\docsvlist{A,B,C,D,E,F,G,H,I,J,K,L,M,N,O,P,Q,R,S,T,U,V,W,X,Y,Z}

\renewcommand{\hat}{\widehat}
\newcommand{\R}{\mathbf{R}}
\newcommand{\C}{\mathbf{C}}
\newcommand{\Q}{\mathbf{Q}}

\newcommand{\K}{\mathbf{K}}

\newcommand{\m}{\mathfrak{m}}

\renewcommand{\div}{\operatorname{div}}

\DeclareMathOperator{\Pic}{Pic}
\DeclareMathOperator{\Div}{Div}
\DeclareMathOperator{\Gal}{Gal}
\DeclareMathOperator{\ord}{ord}
\newcommand{\OO}{\mathcal O}

\DeclareMathOperator{\Supp}{Supp}
\DeclareMathOperator{\supp}{Supp}

\renewcommand{\mod}{\text{mod}}

\DeclareMathOperator{\spec}{Spec}

\DeclareMathOperator{\an}{an}

\DeclareMathOperator{\vol}{vol}

\newtheorem{thm}{Theorem}[section]
\newtheorem{thm*}{Theorem}
\newtheorem{bigthm}{Theorem}
\newtheorem{prop}[thm]{Proposition}
\newtheorem{cor}[thm]{Corollary}
\newtheorem{lemme}[thm]{Lemma}

\theoremstyle{definition}
\newtheorem{dfn}[thm]{Definition}
\newtheorem{ex}[thm]{Example}

\AtBeginEnvironment{dfn}{%
  \setlist[enumerate]{label={(\roman*)}}
}

\AtBeginEnvironment{thm}{%
  \setlist[enumerate,1]{label={(\arabic*)}}
}

\AtBeginEnvironment{prop}{%
  \setlist[enumerate,1]{label={(\arabic*)}}
}

\AtBeginEnvironment{lemme}{%
  \setlist[enumerate,1]{label={(\arabic*)}}
}

\begin{document}
\author{Marc Abboud \\ Université de Neuchâtel}
\title{A local version of the arithmetic Hodge index theorem over quasiprojective varieties}
\address{Marc Abboud, Institut de mathématiques, Université de Neuchâtel
\\ Rue Emile-Argand 11 CH-2000 Neuchâtel}
\email{marc.abboud@normalesup.org}
\thanks{The author acknowledge support by the Swiss National Science Foundation Grant “Birational transformations of higher dimensional varieties” 200020-214999.}

\begin{abstract}{We define a local intersection number for metrised line bundles over quasiprojective varieties with
    compact support and show the local arithmetic Hodge index theorem for this intersection number.  As a consequence we
  obtain a uniqueness result for the Monge-Ampère equation over quasiprojective varieties within a certain class of
solutions both in the archimedean and non-archimedean setting.}
\end{abstract}
\maketitle

\section{Introduction}\label{sec:intro}
\subsection{Local intersection number for metrised line bundles over quasiprojective
varieties}\label{subsec:strongly-compactly-supported-intro}
Let $\K_v$ be a complete field and let $U$ be a normal quasiprojective variety over $\K_v$. In
\cite{yuanAdelicLineBundles2023} Yuan and Zhang have defined metrised line bundles and arithmetic divisors over
$U$ as limits of such objects on projective compactification of $U$ with certain properties.

If $X$ is a projective variety over $\K_v$ and $\overline M, \overline L_1, \dots, \overline L_n$ are integrable
metrised line bundles with $M$ the trivial line bundle (we say that $\overline M$ is \emph{vertical}), then the local
intersection number
\begin{equation}
\overline M \cdot \overline L_1 \cdots \overline L_n
  \label{<+label+>}
\end{equation}
is well defined. Furthermore we have an integration by parts formula: if $s$ is a rational section of $L_n$, then 
\begin{equation}
  \overline M \cdot \overline L_1 \cdots \overline L_{n} = (\overline M \cdot \overline L_1 \cdots \overline
  L_{n-1})_{|\div (s)} - \int_{X^{\an}} \log \| s \|_{\overline L_n} c_1 (\overline M) \cdot c_1 (\overline L_1) \cdots
  c_1 (\overline L_{n-1}).
  \label{<+label+>}
\end{equation}
In the quasiprojective setting, there is no yet definition of a local intersection number. If $\overline M$ is vertical
and $\overline L_i$ are integrable line bundles then one would like to define the following intersection number 
\begin{equation}
  \overline M \cdot \overline L_1 \cdots \overline L_n := \int_{U^{\an}} g_{\overline M} c_1 (\overline L_1) \cdots c_1
  (\overline L_n).
  \label{eq:def-intersection-number}
\end{equation}
But this poses many problems. First, it is not clear that $g_{\overline M}$ is integrable with respect to $c_1
(\overline L_1) \cdots c_1 (\overline L_n)$, nor that this is compatible with the limit process in the definition of
Yuan and Zhang.  However, if $g_{\overline M}$ is constant outside a compact subset then
\eqref{eq:def-intersection-number} makes sense. We call such metrised line bundles \emph{compactly supported vertical
line bundels}. We also show that the integration by parts formula holds. 

\begin{bigthm}\label{bigthm:intersection-number}
  Let $U$ be a normal quasiprojective variety over a complete field $\K_v$. Let $\hat \Pic (U)_{int}$ be the set of
  integrable metrised line bundles over $U$ and let $\hat \Pic_c (U)$ be the set of integrable compactly supported
  vertical metrised line bundles. There is a well defined local intersection number 
  \begin{equation}
    \hat \Pic_c (U) \times \hat \Pic(U)_{int} \rightarrow \R
    \label{<+label+>}
  \end{equation}
  defined by \eqref{eq:def-intersection-number}. It is linear in the first variable and multilinear and symmetric in the
  last $n$ variables. Furthermore, there is an integration by parts formula: if $s$ is a rational
  section of ${L_n}_{|U}$, then 
  \begin{equation}
    \overline M \cdot \overline L_1 \cdots \overline L_{n} = \left( \overline M \cdot \overline L_1 \cdots \overline
    L_{n-1} \right)_{|\div (s)} - \int_{U^{\an}} \log \|s \|_{\overline L_n} c_1 (\overline M) \cdot c_1 (\overline L_1)
    \cdots c_1 (\overline L_{n-1}).
    \label{<+label+>}
  \end{equation}
\end{bigthm}
Furthermore we prove an arithmetic Hodge index theorem for this class of vertical line bundles.

The proof of this theorem relies on the following density results. If $\Omega ' \Subset U^{\an}$ is a compact
subset, then there exists a compact subset $\Omega \Subset U^{\an}$ containing $\Omega '$ such that the set of
model functions over $U^{\an}$ which are constant outside $\Omega$ is dense in the set of continuous functions that are
constant outside $\Omega$. To try to define an intersection number for a larger class of vertical line bundles would
require a deeper study of integrability of Green functions over quasiprojective varieties which is a difficult problem
even over $\C$. 

\subsection{An arithmetic Hodge index theorem over quasiprojective varieties}\label{subsec:arithm-hodge-index-intro}

The arithmetic Hodge index theorem proven in \cite{yuanArithmeticHodgeIndex2017} states the following. Let $X$ be
a projective variety of dimension $n$ over a complete field, let $\overline M$ be a vertical integrable line bundle and
let $\overline L_1, \dots, \overline L_{n-1}$ be semipositive line bundles with $L_i$ big and nef, then
\begin{equation}
  \overline M^2 \cdot \overline L_1 \cdots \overline L_{n-1} \leq 0.
  \label{<+label+>}
\end{equation}
Furthermore if $\overline M$ is $\overline L_i$-bounded for every $i = 1, \dots, n-1$, then we have equality if and only
if the metric of $\overline M$ is constant. Using our newly defined intersection number, we show that this result holds
in the quasiprojective setting if $\overline M$ is integrable and compactly supported. 

\begin{bigthm}\label{bigthm:arithmetic-hodge-index}
  Let $U$ be a quasiprojective variety over a complete field $\K_v$. Let $\overline M$ be a compactly
  supported vertical line bundle and let $\overline L_1, \dots, \overline L_{n-1}$ be nef metrised line bundles
  over $U$, then
  \begin{equation}
    \overline M^2 \cdot \overline L_1 \cdots \overline L_{n-1} \leq 0.
    \label{<+label+>}
  \end{equation}
  Furthermore, if $\overline M$ is $\overline L_i$-bounded for every $i$ and $L_i^n > 0$ then we have equality if and
  only if the metric of $\overline M$ is constant over $U^{\an}$.
\end{bigthm}
The positivity assumption for the $L_i$'s is necessary for the equality statement. Indeed, we can for example write the
line bundle $\overline 0 = \lim_k \frac{1}{k} \overline A$ where $\overline A$ is any semipositive line bundle. As in
\cite{yuanArithmeticHodgeIndex2017} we have the following corollary which is a variant of the theorem of Calabi over
Kahler varieties.

\begin{bigthm}\label{bigthm:calabi}
  Let $\K_v$ be a complete closed field and $U$ be a quasiprojective variety over $\K_v$ of dimension $n$. Let $\overline
  L_1, \overline L_2$ be two nef metrised line bundles such that $c_1 (\overline L_1)^n = c_1 (\overline L_2)^n$ and $\overline
  L_1 - \overline L_2$ is vertical and compactly supported (in particular $L_1,L_2$ have the same underlying line bundle
$L$). If $L^n > 0$, then the metric of $\overline L_1 - \overline L_2$ is constant.
\end{bigthm}

We believe these results could have applications in the study of quasiprojective algebraic
dynamical systems. The author uses Theorem \ref{bigthm:calabi} in \cite{abboudRigidityPeriodicPoints2024} to show the
following result. If $U$ is a normal affine surface over an algebraically closed field $K$ and $f,g$ are automorphisms
with first dynamical degree $>1$, then $f,g$ share a Zariski dense set of periodic points if and only if they have the
same periodic points.

\subsection*{Acknowledgments}\label{subsec:acknowledgements}
Part of this paper was written during my visit at Beijing International Center for Mathematical Research which
I thank for its welcome. I thank Junyi Xie and Xinyi Yuan for our dicussions on adelic divisors and line bundles during
this stay.

\section{Arithmetic divisors and metrised line bundles}\label{sec:arithm-divisors-metrised-line-bundles}
\subsection{Terminology}\label{sec:terminology}
Let $R$ be an integral domain. A variety over $R$ is a flat, normal, integral, separated scheme over $\spec R$ locally of
finite type. If $R = \K$ is a field we require in addition that a variety over $\K$ is geometrically irreducible.
\subsection{Divisors and line bundles}\label{subsec:divisors}
Let $X$ be a normal $R$-variety, a \emph{Weil divisor} over $X$ is a formal sum of irreducible codimension 1 closed
subvarieties of $X$ with integer coefficients. A \emph{Cartier divisor} over $X$ is a global section of the sheaf
$\mathcal K_X^\times / \OO_X^\times$ where $\mathcal K_X$ is the sheaf of rational functions over $X$ and $\OO_X$ the
sheaf of regular functions over $X$. If $R$ is Noetherian, every Cartier divisor induces a Weil divisor. If $X$ is a
projective variety over a field $K$ and  a Weil divisor over $X$, we write $\Gamma (X,D)$ for the
set
\begin{equation}
  \Gamma (X,D) = \left\{ f \in K (X)^\times : \div (f) + D \geq 0 \right\}.
  \label{eq:<+label+>}
\end{equation}
A \emph{$\Q$-Cartier divisor} is a formal sum $D = \sum_i \lambda_i D_i$ where $\lambda_i \in \Q$ and the $D_i$'s are
Cartier divisors.
If $X$ is a projective variety over $R$, we write $\OO_X(D)$ for the line bundle associated to $D$.
If $R = \K$ is a field and $L$ is a line bundle over $X$. We write $H^0 (X,L)$ for the space of global sections of $L$ and
$h^0:= \dim_\K H^0 (X, L)$. The \emph{volume} of $L$ is
\begin{equation}
  \vol (L) = \lim_{m \rightarrow + \infty} \frac{h^0 (X, mL)}{m^n / n!}.
  \label{eq:<+label+>}
\end{equation}
It is also defined for $\Q$-line bundles (see \cite{lazarsfeldPositivityAlgebraicGeometry2004}). A line bundle $L$ is
\emph{big} if $\vol(L) > 0$ and if $L$ is nef, then $\vol (L) = L^{\dim X}$.
We have the following inequality (see e.g Example 2.2.33 of \cite{lazarsfeldPositivityAlgebraicGeometry2004})
\begin{thm}\label{thm:inequality-volume}
  If $X$ is a projective variety of dimension $n$ over a field $K$ and $L, M$ are nef $\Q$-line bundles, then
  \begin{equation}
    \vol (L - M) \geq L^{n} - n \cdot L^{n-1}\cdot M.
    \label{eq:<+label+>}
  \end{equation}
\end{thm}
\subsection{Berkovich spaces}\label{subsec:berkovich-spaces} For a general reference on Berkovich spaces, we refer to
\cite{berkovichSpectralTheoryAnalytic2012}.
Let $\K_v$ be a complete field with respect to an absolute value $| \cdot |_v$. If $X$ is a quasiprojective
variety over $\K_v$, we write $X^{\an}$ for the Berkovich analytification of $X$ with respect to $\K_v$. It is a locally ringed space with a contraction map
\begin{equation}
  c : X^{\an} \rightarrow X.
  \label{<+label+>}
\end{equation}
 It is a
Hausdorff space. In
particular, if $X$ is proper (e.g projective), then $X^{\an}$ is compact.

Let $\overline \K_v$ be an algebraic closure of $\K_v$. The absolute value $\left| \cdot \right|_v$ extends naturally to
$\overline \K_v$. If $p \in X(\overline \K_v)$ is a rational point, then it defines a point in
$X^{\an}$. We thus have a map
\begin{equation}
  \iota_0: X (\overline \K_v) \rightarrow X^{\an}
  \label{<+label+>}
\end{equation}
and we write $X (\overline \K_v)$ for its image. It is a dense subset of $X^{\an}$. This map is generally
not injective as two points $p,q \in X(\overline \K_v)$ define the same seminorm if and only if they are in the
same orbit for the action of the Galois group $\Gal (\overline \K_v / \K_v)$.

If $\phi : X \rightarrow Y$ is a morphism of varieties, then there exists a unique morphism
\begin{equation}
  \phi^{\an} :
  X^{\an} \rightarrow Y^{\an}
\end{equation}
such that the diagram
\begin{equation}
\begin{tikzcd}
  X^{\an} \ar[r, "\phi^{\an}"] \ar[d] & Y^{\an} \ar[d] \\
  X \ar[r, "\phi"] & Y
  \label{<+label+>}
\end{tikzcd}
\end{equation}
commutes. In particular, if $X \subset Y$, then $X^{\an}$ is isomorphic to $c_Y^{-1}(X)
\subset Y^{\an}$.

If $\K_v$ is not algebraically closed, let $\C_v$ be the completion of the algebraic closure of $\K_v$ with respect to
$v$. If $X$ is a variety over $\K_v$ and $X_{\C_v}$ is the base change to $\C_v$, then we have the following
relation for the Berkovich space
\begin{equation}
  X^{\an} = X_{\C_v}^{\an} / \Gal(\overline \K_v / \K_v)
  \label{eq:<+label+>}
\end{equation}
and the continuous map $X_{\C_v}^{\an} \rightarrow X^{\an}$ is proper (the preimage of a compact subset is a
compact subset) if $X$ is quasiprojective.
In particular, if $\K_v = \R$ and $X_\R$ is a variety over $\R$, then
\begin{equation}
  X_\R^{\an} = X_\R (\C) / (z \mapsto \overline z).
  \label{eq:<+label+>}
\end{equation}

\subsection{Arithmetic divisors and metrised line bundles}
Let $\K_v$ be a complete with respect to an absolute value $|\cdot|_v$. If $v$ is
archimedean, then $\K_v = \C$ or $\R$ and we set $\OO_v = \K_v$. Otherwise let $\OO_v$ be the
valuation ring of $\K_v$ and let $\kappa_v = \OO_v / \m_v$ be the residual field of $\OO_v$ where $\m_v$ is
the maximal ideal of $\OO_v$ of elements of absolute value $<1$. If $v$ is non archimedean, then $\OO_v$ is
Noetherian if and only if $v$ is discrete.

Let $X$ be a projective variety over $\K_v$. If $D = \sum a_i D_i$ is a $\Q$-Cartier divisor over $X$, a
\emph{Green} function of $D$ is a continuous function $g : X^{\an} \setminus
(\Supp D)^{\an} \rightarrow \R$ such that for every point $p \in (\Supp D)^{\an}$, if $z_i$ is a local
equation of $D_i$ at $p$ then the function
\begin{equation}
  g + \sum_i a_i \log |z_i|
  \label{<+label+>}
\end{equation}
extends locally to a continuous function at $p$. Notice that compared to
\cite{moriwakiAdelicDivisorsArithmetic2016} or \cite{yuanAdelicLineBundles2023}, our definition of Green
functions differs by a factor 2.
An \emph{arithmetic} divisor over $X$ a pair $\overline D = (D, g)$ where $D$ is a $\Q$-Cartier divisor and $g$ is a
Green function of $D$.
An arithmetic divisor is called \emph{principal} if some of its multiple is of the form
\begin{equation}
  \label{eq:39}
  \hat \div (P) := (\div(P), -\log |P|)
\end{equation}
where $P$ is a rational function over $X$.
We write $\hat \Div(X)$ for the space of arithmetic divisor and
$\hat \Pr (X )$ for the set of principal arithmetic divisors. If $\K_v = \R$, then a Green function
of $D$ over $X_\R^{\an}$ is the same as the data of a Green function of $D_\C$ over $X_\R (\C)$ invariant by complex conjugation.

We say that an arithmetic divisor $\overline D = (D,g)$ is \emph{effective}, written $\overline D \geq 0$ if
$g \geq 0$, this implies in particular that $D$ is an effective divisor. We write
$\overline D \geq \overline D'$ for $\overline D - \overline D' \geq 0$.

A \emph{metrised} line bundle $\overline L$ over $X$ is a pair given by a $\Q$-line bundle $L$ over $X$, a positive
integer $m$ such that $mL$ is a line bundle and a
family of metrics $|\cdot|_x$ of the stalk of $m \cdot L^{\an}$ at $x$ for $x \in X^{\an}$ such that for every
open subset $U \subset X$ and $s \in H^0 (U, mL)$, the function
\begin{equation}
  \label{eq:40}
  x \in U^{\an} \mapsto |s(x)|_x
\end{equation}
is continuous. An isometry of metrized line bundle is a map $\phi : \overline mL \rightarrow \overline m' L'$
where $\phi : mL \rightarrow m' L'$ is an isomorphism of line bundles and such that
$| \cdot|_{\overline{m' L '}} = |\cdot|_{\overline{mL}} \circ \phi^*$. We write $\hat \Pic(X )$ for the set of
  metrized line bundles over $X$ modulo isometries.

  If $\overline M$ is a metrized line bundle with a trivial underlying line bundle, then we define the \emph{Green}
  function of $\overline M$ as $g_{\overline M} := - \log |1| : X^{\an} \rightarrow \R$.

If $L = \OO_{X} (D)$ where $D$ is a $\Q$-Cartier divisor over $X$ and $\overline D = (D,g)$ is an
arithmetic divisor, then we can define the metrized line bundle $\overline L =: \OO_{X} (\overline D)$ by the
following procedure. If $s_D$ is the trivialising section of $\OO_{X} (D)$ over $X \setminus \Supp D$ such
that $\div(s_D) = D$, then we set
\begin{equation}
  \label{eq:41}
  \forall x \not \in (\Supp D)^{\an}, \quad |s_D(x)| = e^{-g(x)}.
\end{equation}
This defines $\overline L$ up to isometry and we have an isomorphism of groups
\begin{equation}
  \label{eq:42}
  \overline D \in \hat \Div (X ) / \hat \Pr (X ) \mapsto \OO_{X}(\overline
  D) \in \hat \Pic(X ).
\end{equation}
The inverse homomorphism is given by for any $\overline L \in \hat \Pic(X )$ and any
rational section $s$ of $mL$ where $mL$ is a line bundle we associate the arithmetic divisor
\begin{equation}
  \label{eq:43}
  \frac{1}{m}\hat \div(s) := (\frac{1}{m}\div(s), \quad x \mapsto -\frac{1}{m} \log |s(x)|_x).
\end{equation}

\subsection{Model Green functions} Suppose that $v$ is non-archimedean. If $D$ is a divisor on
$X$, a \emph{model} of $(X,D)$ is the data of $(\sX, \sD)$ where $\sX$ is a model of
$X$ over $\OO_v$ and $\sD$  is a divisor on $\sX$ such that ${\sD}_{|X} = D$.
If $(\sX, \sD)$ is a model of $(X,D)$, then $(\sX, \sD)$ induces a Green function of
$D$ over $X^{\an}$ as follows. We have the reduction map (see \cite{berkovichSpectralTheoryAnalytic2012})
\begin{equation}
  r_{\sX} : X^{\an}\rightarrow \sX_{\kappa_v} = \sX \times_{\spec
  \OO_v} \spec \kappa_v.
  \label{<+label+>}
\end{equation}
Let $p \in X^{\an} \setminus (\Supp D)^{\an}$, and let $z_i$ be a local equation of $\sD_{i}$ at
$r_{\sX} (p)$ where we have written $\sD = \sum_i a_i \sD_i$ and the $\sD_i$'s are Cartier divisors, we define
\begin{equation}
  g_{(\sX, \sD)} (p) := - \sum_i a_i \log |z_i|.
  \label{<+label+>}
\end{equation}
It does not depend on the choice of the local equations $z_i$. We call such divisors \emph{model arithmetic divisors}.
The reduction map $r_{\sX}$ is always surjective and for every irreducible component $\mathsf X$ of the special
fiber there is a unique $x \in X^{\an}$ such that $r_{\sX}(x)$ is the generic point of $\mathsf X$, these points are
called \emph{Shilov points} and they are dense in $X^{\an}$. 

If $L$ is a $\Q$-line bundle over $X$, a \emph{model} of $(X,L)$ is the data of $(\sX, \sL)$ where
$\sX$ is a projective model of $X$ over $\OO_v$ and $\sL$ is a $\Q$-line bundle over
$\sX$ such that $\sL_{|X} = L$.
A model $(\sX, \sL)$ defines a metrized line bundle over $X$ as follows. Let $m$ be an integer such that $mL$ and
$m \sL$ are line bundles. Let $x \in X^{\an}$, the line
bundle $m\sL$ is locally free of rank 1 at $r_{\sX} (x)$, let $s$ be a generator of the stalk at $r_{\sX}(x)$,
then $mL^{\an}$ is locally generated at $x$ by $r_{\sX}^* (s)$ and we set $|r_{\sX}^* (s) (x)|_x = 1$.
This does not depend on the choice of the local generator $s$. We call such line bundles \emph{model metrised line
bundles}.

A \emph{model} function over $X^{\an}$ is the Green function of a vertical $\Q$-line bundle over
$X$ or equivalently the Green function of a vertical model $\Q$-divisor.

\begin{ex}\label{ex:constant-green-function}
  In particular, every principal arithmetic divisor is a model arithmetic divisor and if $\lambda \in \K_v$, then
  $\hat \div (\lambda)$ is a vertical model arithmetic divisor with constant Green function equal to $\log \left| \lambda
  \right|$.
\end{ex}

\begin{prop}[\cite{gublerLocalHeightsSubvarieties1998} Lemma 7.8]\label{prop:model-functions-stable-by-sum-and-max}
  The set of model functions is a $\Q$-vector space stable by maximum, i.e if $f,g$ are model functions over
  $X^{\an}$, then $\max(f,g)$ also is.
\end{prop}

\begin{prop}[\cite{gublerLocalHeightsSubvarieties1998} Theorem 7.12]\label{prop:density-model-functions-projective}
  If $X$ is a projective variety over $\K_v$, then model functions are dense in the set of continuous functions
  from $X^{\an}$ to $\R$ with respect to the uniform topology.
\end{prop}

If $v$ is archimedean, we make the convention that every arithmetic divisor is a model arithmetic divisor and every
metrised line bundle is a model one.

\subsection{Positivity}
\label{subsec:positivity}
Let $X$ be a projective variety over $\K_v$. If $v$ is archimedean, then we have the usual notion of smooth
and plurisubharmonic functions over $X_\C(\C)$ and of closed positive current. If $v$ is non-archimedean,
Chambert-Loir and Ducros developped a theory of smooth and plurisubharmonic functions and closed positive current over
$X^{\an}$ in \cite{chambert-loirFormesDifferentiellesReelles2012} which mimic the complex setting. In particular, the
operator $dd^c$ is defined in the non-archimedean setting and we have the Poincaré-Lelong formula: If $P$ is a rational
function over $X$, then
\begin{equation}
  dd^c (\log \left| P \right|) = \delta_{\div(P)}.
  \label{eq:<+label+>}
\end{equation}
Therefore, if $\overline M$ is a metrised line bundle over $X$ with a continuous metric, $U$ is an open subset of
$X^{\an}$ and $s$ is a regular section of $M$ over $U$, then we have the equality of currents.
\begin{equation}
  c_1 (\overline M) = \delta_{\div(s)} + dd^c \log \left| \left| s \right| \right|^{-1}.
  \label{eq:<+label+>}
\end{equation}
Let $\overline L$ be a metrized line bundle over $X$. If $v$ is archimedean, we say that $\overline L$ is
\emph{semipositive} if $g$ is plurishubarmonic.
If $v$ is non-archimedean, we say that $\overline L$ is \emph{semipositive} if there exists a sequence of
models $(\sX_n, \sL_n)$ of $(X,L)$ such that $\sL_n$ is nef and $d_\infty(\sL_n, \overline L) \rightarrow 0$. In
particular a model metrised line bundle $(\sX, \sL)$ is semipositive if and only if it is nef.

We say that a metrised line bundle is \emph{integrable} if it is the difference of two semipositive line bundles.

All these notions have analog for arithmetic divisors $\overline D$ by checking the definition on
$\OO_X(\overline D)$.

\subsection{Chambert-Loir measures}\label{subsec:chambert-loir-measures-proj}
Let $X$ be a projective variety of dimension $n$ over $\K_v$, let $\overline L_1, \cdots, \overline L_n$ be
integrable metrised line bundles. By the theory of Bedford and Taylor in the archimedean case and by
\cite{chambert-loirFormesDifferentiellesReelles2012} \S5 in the non-archimedean case. We can define the current
$c_1(\overline L_i)$ and the measure
\begin{equation}
  \label{eq:45}
  c_1(\overline L_1) \cdots c_1(\overline L_n) := c_1(\overline L_1) \wedge \cdots \wedge c_1(\overline L_n)
\end{equation}
over $X^{\an}$.

If $\K_v = \R$, then we make the base change to $\C$ to define the Chambert-Loir measure over $X_\C (\C)$ using
Bedford-Taylor theory and then pushforward the measure to $X_\R^{\an}$.

If $\overline D$ is an arithmetic divisor, then we set $c_1(\overline D) = c_1 (\OO_X (\overline D))$. This is
well defined as for any rational function $P$, $c_1 (\OO_{X} (\hat \div(P))) = 0$.
\begin{prop}[Corollaire 6.4.4 of \cite{chambert-loirFormesDifferentiellesReelles2012}]
  \label{prop:1}
  If all the $\overline L_i$'s are semipositive, then the measure
$\mu := c_1(\overline L_1) \cdots c_1(\overline L_n)$ is a positive measure of total mass
$L_1 \cdots L_n$. Furthermore, the measure does not charge any analytic subvarieties, that is $\mu(Y^{\an}) = 0$ for
every closed strict subvariety $Y \subset X$.
\end{prop}

\subsection{Local intersection number}
\label{subsec:local-inters-numb}
For the definition of the local intersection number we follow \S 3.6 of \cite{chenArithmeticIntersectionTheory2021}.
If $\overline M$ is an integrable metrized line bundle with trivial underlying line bundle and
$\overline L_1, \cdots, \overline L_n$ are integrable metrized line bundle then we define the local intersection number
\begin{equation}
  \label{eq:63}
  \overline M \cdot \overline L_1 \cdots \overline L_n := \int_{X^{\an}} g_{\overline M} c_1 (\overline L_1) \cdots c_1 (\overline L_n).
\end{equation}
And we have the following integration by parts formula, suppose $s_n$ is a rational section of $L_n$, then
\begin{equation}
  \overline M \cdot \overline L_1 \cdots \overline L_n = \left( \overline M \cdot \overline L_1 \cdots \overline
  L_{n-1} \right)_{\div(s_n)} - \int_{X^{\an,v}} \log \left| \left| s_n \right| \right|_{\overline L_n} c_1 (\overline M) c_1
  (\overline L_1) \cdots c_1 (\overline L_{n-1}).
  \label{eq:<+label+>}
\end{equation}
This intersection product is multilinear, symmetric and invariant by extension of scalars.

\subsection{Arithmetic divisors and metrised line bundles over quasiprojective varieties.}\label{subsec:local-setting-quasiprojective-varieties}
The main reference for this section is \cite{yuanAdelicLineBundles2023} \S 3.6.
Let $\K_v$ be a complete field and $U$ be a quasiprojective variety over $\K_v$.
A \emph{model} arithmetic divisor over $U$ is a model arithmetic divisor $\overline \sD$ over any projective
model $X$ of $U$ over $\K_v$ such that $\sD_{|U}$ is a Cartier divisor. That is the set $\hat\Div
(U )_\mod$ of model arithmetic divisor over $U$ is defined as
\begin{equation}
  \hat\Div (U )_\mod := \varinjlim_{X \supset U} \hat\Div (X )_\mod.
  \label{eq:<+label+>}
\end{equation}

If $v$ is archimedean, a \emph{boundary divisor} over $U$ is a model arithmetic divisor $\overline \sD_0$ defined
on some projective model $X$ with Green function $g_0$ such that $\Supp D_0 = X \setminus U$ and
$g_0 > 0$. If $v$ is non-archimedean a \emph{boundary divisor} is a model arithmetic divisor $\overline \sD_0$ such that
$\sD_0$ is an effective Cartier divisor.
We have an
\emph{extended norm} on $\hat\Div (U )_\mod$ defined by
\begin{equation}
  \left| \left| \overline \sD \right| \right|_{\overline D_0} := \inf \left\{ \epsilon >0 : - \epsilon \overline \sD_0
  \leq \overline \sD \leq \epsilon \overline \sD_0 \right\}.
  \label{eq:<+label+>}
\end{equation}
An \emph{arithmetic divisor} $\overline D$ over $U$ is an element of the completion of $\hat\Div (U )_\mod$
with respect to $\left| \left| \cdot \right| \right|_{\overline D_0}$. More precisely, it is given by a sequence
$(\sX_n, \overline \sD_n)$ of model arithmetic divisors such that there exists a sequence of positive rational numbers
$\epsilon_n \rightarrow 0$ such that
\begin{equation}
  \forall m \geq n, \quad -\epsilon_n \overline \sD_0 \leq \overline \sD_m - \overline \sD_n \leq \epsilon_n \overline \sD_0.
  \label{eq:inequality-model-divisors}
\end{equation}
Such a sequence of model arithmetic divisors is called a \emph{Cauchy sequence}.
Notice that for every $m, n, {\sD_m}_{|U} = {\sD_n}_{|U} =: D$. In terms of Green function, this is
equivalent to checking that over $\left( U \setminus \Supp D \right)^{\an}$,
\begin{equation}
  - \epsilon_n g_0 \leq g_m - g_n \leq \epsilon_n g_0.
  \label{eq:<+label+>}
\end{equation}
In particular, the sequence $g_n$ converges uniformly locally towards a continuous function $g_{\overline D}$ which we call the
\emph{Green function} of $\overline D$. The topology and the completed space do not depend on the choice of $\overline
D_0$. We write $\hat\Div (U )$ for the set of arithmetic divisors over $U$.
For metrised line bundles over $U$ the definitions are analogous. We write $\hat \Pic(U )$ for
the set of metrised line bundles over $U$. In particular if $\overline L_i = \lim_k \overline L_{i,k}$, then the
geometric intersection number 
\begin{equation}
  L_{1,k} \cdots L_{n,k}
  \label{<+label+>}
\end{equation}
converges towards a number that we write 
\begin{equation}
  L_1 \cdots L_n.
  \label{<+label+>}
\end{equation}

\begin{dfn}
An arithmetic divisor $\overline D$ over $U$ is
\begin{itemize}
  \item \emph{vertical} if it is the limit of vertical model arithmetic divisors.
  \item \emph{strongly nef} if for the Cauchy sequence $(\overline \sD_i)$ defining
it we can take for every $\overline \sD_i$ a semipositive model arithmetic divisor.
\item \emph{nef} if there exists a strongly nef arithmetic divisor $\overline A$ such that for all $m \geq 1, m\overline
  D + \overline A $ is strongly nef.
\item \emph{integrable} if it is the difference of two strongly nef arithmetic divisors.
\end{itemize}
and we have analog definitions for metrised line bundles.
\end{dfn}

\subsection{Chambert-Loir measures}
If $\overline L_1, \dots, \overline L_n$ are integrable metrised line bundles over $U$ a quasiprojective variety over
$\K_v$ of dimension $n$, then we define the Chambert-Loir measure
\begin{equation}
  c_1 (\overline L_1) \cdots c_1 (\overline L_n)
  \label{eq:<+label+>}
\end{equation}
over $U^{\an}$ as the weak limit of the measures
\begin{equation}
  c_1 (\overline L_{1,i}) \cdots c_1 (\overline L_{n,i})
  \label{eq:<+label+>}
\end{equation}
where $(\overline L_{k,i})_{i \geq 1}$ is any Cauchy sequence defining $\overline L_k$. If all the $\overline L_i$'s are
nef, then this is a positive measure, see \cite{yuanAdelicLineBundles2023}, \S 3.6.7. We can still compute the total
mass of this measure.

\begin{thm}[Theorem 1.2 of \cite{guoIntegrationFormulaChern2025}]\label{thm:total-mass-measure-quasiprojective-setting}
  Let $\overline L_1, \dots, \overline L_n$ be integrable line bundles over $U$, then
  \begin{equation}
    \int_{U^{\an}} c_1 (\overline L_1) \cdots c_1 (\overline L_n) = L_1 \cdots L_n.
    \label{eq:<+label+>}
  \end{equation}
\end{thm}

\section{Definition of the intersection number}\label{sec:intersection number}
  \subsection{Density results}
  Let $U$ be a quasiprojective variety over $\K_v$ and let $X$ be a projective model of $U$ over $\K_v$.
  Let $\Omega \Subset U^{\an}$ be a compact subset. We write $\mathbf M_\Omega (X^{\an})$ for the $\Q$-vector space
  generated by model functions over $X^{\an}$ which are constant outside $\Omega$.  

\begin{thm}\label{thm:density-model-functions-with-compact-support}
  Let $U$ be a quasiprojective variety over $\K_v$ and let $\Omega ' \Subset U^{\an}$ be a compact subset.
  For any projective model $X$ of $U$ over $\K_v$, there exists a compact subset $\Omega \subset U^{\an}$ containing
  $\Omega '$ such that $M_\Omega (X^{\an})$ is dense in $\sC^0_\Omega (U^{\an}; \R)$ with respect to the supremum norm
  where $\sC^0_\Omega (U^{\an}; \R)$ is the set of real continuous function over $U^{\an}$ which are constant outside
  $\Omega$.
\end{thm}
When $v$ is archimedean, the theorem is easy. Indeed, let $W \subset U^{\an}$ be a relatively compact connected open subset that
contains $\Omega '$, then $\Omega = \overline W$ works because by the Stone-Weierstrass theorem for locally compact
subspace (see \cite[44D]{willardGeneralTopology2004}), the set of smooth real functions over $W$ vanishing on $\Omega
\setminus W$ is dense in the set of continuous function over $W$ vanishing on $\Omega \setminus W$, adding constant
functions to this set, we get the result for archimedean places. 

Suppose now that $v$ is non-archimedean. We are going to apply a mix of the boolean version of the Stone-Weierstrass
theorem and the Stone-Weierstrass theorem for locally compact spaces. We make the following definition:
   If $Z$ is a scheme and $Y_1, Y_2$ are closed subsets, by blowing up along $Y_1 \cap Y_2$
  we mean blowing up $Z$ along the ideal sheaf $I_{Y_1} + I_{Y_2}$ and then normalise. In the blow up, the strict
  transform of $Y_1$ and $Y_2$ do not meet anymore.
  Start with the following result.
  \begin{prop}\label{prop:bon-model}
    There exists a projective model $\sX$ of $X$ over $\OO_v$ such that there is a closed subset $\sW \subset \sX$ of the
    special fiber with $\Omega ' \subset r_{\sX}^{-1} (\sW)$ and 
      \begin{equation}
        \sW \cap \overline{X \setminus U} = \emptyset.
        \label{<+label+>}
      \end{equation}
  \end{prop}
  \begin{proof}
    First we can assume that there exists a model $\sX$ such that for any $y \in \Omega ', r_{\sX} (y) \not \in
    \overline{X \setminus U}$. Indeed, otherwise this would imply the following: for any
    model $\sX$ of $X$ over $\OO_v$ and for any vertical Cartier divisor $\sD$ over $\sX$ there exists $y \in \Omega '$
  and $z \in \Delta^{\an}$ such that $g_{(\sX,\sD)} (y) = g_{(\sX,\sD)} (z)$ where $\Delta = X \setminus U$. But by
density of the model functions (Proposition \ref{prop:density-model-functions-projective}) and compactness of $\Omega '$
and $\Delta^{\an}$ this would imply the same statement for any continuous function over $X^{\an}$. This is absurd since
$\Omega ' \cap \Delta^{\an} = \emptyset$. 

Now, for any $x \in \Omega '$, define $\sW_x = \overline{r_{\sX}(x)}$ and $W_x = r_{\sX}^{-1}(\sW_x)$. This is an open
subset of $U^{\an}$ and 
\begin{equation}
  \Omega ' \subset \bigcup_{x \in \Omega '} W_x.
  \label{ }
\end{equation}
By compactness, there exists $x_1, \dots, x_r \in \Omega '$ such that $\Omega ' \subset W_{x_1} \cup \cdots \cup
\sW_{x_r}$. Write $\sW_i := \sW_{x_i}$ and $\sW_i = \sW_{x_i}$. Suppose that $\sW_1 \cap \overline \Delta \neq \emptyset$. Then, we
replace $\sX$ by the normalised blow up along $\sZ =\sW_1 \cap \overline \Delta$. Now, by assumption, $r_{\sX}^{-1} (\sZ) \cap
\Omega ' = \emptyset$ and now in this new model $\sX$ we have that $\overline{r_{\sX} (x_1)} \cap \overline \Delta =
\emptyset$ because the strict transform do not intersect anymore after blowing up. Doing this for every $x_i$ we get the
desired model $\sX$, indeed we define $\sW = \sW_1 \cup \cdots \cup \sW_r$.
\end{proof}

Define $\Omega \Subset U^{\an}$ to be the closure of $W := r_\sX^{-1} (\sW)$ where $\sX$ and $\sW$ come from Proposition
\ref{prop:bon-model}. It is a compact subset because
$W \subset U^{\an} \setminus r_\sX^{-1} (\overline \Delta)$. We show the following lemma, the proof of which was suggested to the author by Xinyi Yuan.
  \begin{lemme}\label{lemme:compact-model-functions-separate-points}
    Let $x,y \in W$ with $x \neq y$, then there exists a projective model $\sX'$ of $X$  over
    $\OO_v$ with a morphism $\pi : \sX'
    \rightarrow \sX$ and a vertical Cartier divisor $\sD$ over $\sX'$ such that $g = g_{(\sX', \sD)}$ vanishes outside
    $W$ and $g(x) \neq g(y)$.
  \end{lemme}
  \begin{proof}
    Let $x, y \in W$ with $x \neq y$. Since model functions separates points, there exists a model $\sZ$ of $X$ such
    that $r_\sZ (x) \neq r_\sZ (y)$. Let $\sX '$ be a model of $X$ that dominates both $\sZ$ and $\sX$ and write
    $\pi: \sX ' \rightarrow \sX$ for the morphism.
    We get $r_{\sX '} (x) \neq r_{\sX '} (y)$ and 
    \begin{equation}
      W = r_{\sX'}^{-1} (\pi^{-1} (\sW)).
      \label{<+label+>}
    \end{equation} 
    Now over $\sX '$ we do the following procedure.
    \begin{enumerate}
      \item \label{item:step1} Blow up $\sX'$ along $\overline{r_{\sX'}(x)} \cap
        \overline{r_{\sX'}(y)}$ and let $\sX^{(1)}$ be the newly obtained projective model.
      \item \label{item:step3} Blow up $\sX^{(1)}$ along $\overline{r_{\sX^{(1)}}(x)}$ and let $\sE$ be the (possibly
        reducible) exceptional divisor. Let $\sX^{(2)}$ be the newly obtained projective model of $X$. The divisor
        $\sE$ might not be irreducible but it is a Cartier divisor over $\sX^{(2)}$ by property of the blow-up.
    \end{enumerate}
    Notice that the center of all these blowups are on the special fiber and are contained in (the preimage of)
    $\sW$, hence $\sX^{(1)}$ and $ \sX^{(2)}$ are projective models of $X$ over $\OO_v$ and we still have $W =
    r_{\sX^{(2)}}^{-1} (\pi^{-1} (\sW))$.
    We have that $g_\sE (x) > 0$ because $r_{\sX^{(2)}}(x) \in \Supp \sE$ and step \eqref{item:step1} guarantees that
    $g_\sE(y) = 0$. Finally we have that $g_{\sE}$ vanishes outside $W$ because $\sE$ is supported over $\pi^{-1}
    (\sW)$.
  \end{proof}
  We can now finish the proof of Theorem \ref{thm:density-model-functions-with-compact-support}. Let $\Omega = \overline W$ and define $\Omega^A := \Omega / \partial \Omega$. That is we define the equivalence
  relation over $\Omega$ such that any two points on the boundary of $\Omega$ are equivalent and $\Omega^A$ is the
  quotient of this equivalence relation equipped with the quotient topology. Equivalently, $\Omega^A$ is the Alexandroff
  compactification of $W$, the interior of $\Omega$. In particular, $\Omega^A$ is compact
  Hausdorff. Now we have a homeomorphism 
  \begin{equation}
    \sC^0_\Omega (U^{\an}, \R) \rightarrow \sC^0 (\Omega^A, \R)
    \label{<+label+>}
  \end{equation}
  both equipped with the supremum norm. Indeed, every continuous function which is constant equal to $\lambda \in \R$
  outside $\Omega$ must be equal to $\lambda$ also on $\partial \Omega$.
  Now let $F$ be the image of $\mathbf M(\Omega, X^{\an})$ in $\sC^0 (\Omega^A, \R)$. We have that $F$ satisfies
  \begin{enumerate}
    \item \label{item:constant-functions} There exists a constant function in $F$.
    \item \label{item:separation-of-points} For any $x,y \in \Omega^A, \exists f \in F, f(x) \neq f(y)$.
    \item \label{item:Q-vector-space} $F$ is a $\Q$-vector space.
    \item \label{item:stable-by-max} For any $f,g \in F, f+g, \max(f,g) \in F$.
  \end{enumerate}
  Indeed, \ref{item:constant-functions} holds because $\mathbf M (\Omega, X^{\an})$ contains constant functions. Item
  \ref{item:separation-of-points} follows from
  Lemma \ref{lemme:compact-model-functions-separate-points} and Items \ref{item:Q-vector-space} and
  \ref{item:stable-by-max} follow from Proposition
  \ref{prop:model-functions-stable-by-sum-and-max}. By Lemma 3.3.4 of \cite{moriwakiAdelicDivisorsArithmetic2016}, we
  have that $A$ is dense in $\cC^0 (\Omega^A, \R)$ and the theorem is proven.

\subsection{Intersection number}
We now define the intersection number. We say that a vertical line bundle $\overline M$ is \emph{compactly supported} if
$g_{\overline M}$ is constant outside a compact subset. We define the \emph{support} of $\overline M$ as the largest
connected compact subset $\Omega$ of $U^{\an}$ such that $g_{\overline M}$ is constant outside $\Omega$.
Theorem \ref{thm:density-model-functions-with-compact-support} implies the following. 
\begin{prop}\label{prop:approximation-by-strongly-compactly-supported}
  Let $U$ be a quasiprojective variety over $\C$ and let $\overline M$ be an integrable compactly supported vertical
  line bundle over $U$ then there exists a compact
  subset $\Omega \Subset U^{\an}$ containing the support of $\overline M$ and sequence of model vertical line bundles
  $\overline M_i$ such that $g_{\overline M_i}$ is constant outside $\Omega$ and $g_{\overline M_i} \rightarrow
  g_{\overline M}$ uniformly over $U^{\an}$. In
  particular, we have also the convergence of currents $c_1 (\overline M_i) \rightarrow c_1 (\overline M)$ over
  $U^{\an}$.
\end{prop}
We call such a sequence $(\overline M_i)$ an \emph{approximating} sequence of $\overline M$. If $\overline M$ is an
integrable compactly supported vertical line bundle and $\overline L_1, \dots, \overline L_n$ are
integrable line bundles then we define 
\begin{equation}
  \overline M \cdot \overline L_1 \cdots \overline L_n := \int_{U^{\an}} g_{\overline M} c_1 (\overline L_1) \cdots c_1
  (\overline L_n).
  \label{<+label+>}
\end{equation}
This definition behaves well with approximating sequences.
\begin{thm}\label{thm:local-intersection-number}
  Let $\overline M$ be an integrable compactly supported vertical line bundle over a quasiprojective variety $U$ and
  let $\overline L_1, \dots, \overline L_n$ be integrable metrised line bundles over $U$. For any approximating sequence
  $\overline M_i$ of $\overline M$ we have
  \begin{equation}
    \lim_i \overline M_i \cdot \overline L_{1,i} \cdots \overline L_{n,i} = \overline M \cdot \overline L_1 \cdots
    \overline L_n.
    \label{eq:<+label+>}
  \end{equation}
\end{thm}
\begin{proof}
  Write $\mu_i : = c_1 (\overline L_{1,i}) \cdots c_1 (\overline L_{n,i}) $ and $\mu = c_1 (\overline L_1) \cdots c_1
  (\overline L_n)$. Let $\Omega$ be a compact subset such that for every $i, g_{\overline M_i}$ is constant equal to
  $\lambda_i$ outside of $\Omega$. Then we also have that $g_{\overline M}$ is constant equal to some $\lambda \in \R$
  outside of $\Omega$. We have by definition of the intersection number in the projective case 
  \begin{align}
    \overline M_i \cdot \overline L_{1,i} \cdots \overline L_{n,i} &= \int_{U^{\an}} g_{\overline M_i} c_1 (\overline
    L_{1,i}) \cdots c_1 (\overline L_{n,i}) \\
    &= \int_{\Omega} g_{\overline M_i} c_1 (\overline L_{1,i}) \cdots c_1 (\overline L_{n,i}) + \lambda_i \mu_i
    (U \setminus \Omega).
    \label{<+label+>}
  \end{align}
   Now,
  \begin{align}
    \overline M \cdot \overline L_1 \cdots \overline L_n &= \int_{U^{\an}} g_{\overline M} \mu \\
    &= \int_{\Omega} g_{\overline M} \mu + \lambda \mu (U^{\an} \setminus \Omega).
    \label{<+label+>}
  \end{align}
  Since $\mu$ is the weak limit of $\mu_i$ and by Theorem \ref{thm:total-mass-measure-quasiprojective-setting} we have
  that $\mu_i (U) \rightarrow \mu (U)$ and since $g_{\overline M_i}$ converges uniformly towards $g_{\overline M}$, the
  integral over $\Omega$ converges.
\end{proof}

\begin{prop}\label{prop:weak-limit-integration}
  Let $\overline M$ be an integrable compactly supported vertical line bundle over $U$ with an approximating sequence
  $\overline M_i$. Let $\overline D$ be an arithmetic divisor over $U$, then for every integrable line bundle
  $\overline L_1, \dots, \overline L_{n-1}$, the function $g_{\overline D}$ is $c_1 (\overline M) c_1 (\overline L_1)
  \dots c_1 (\overline L_{n-1})$ integrable and
  \begin{equation}
    \int_{U^{\an}} g_{\overline D} c_1 (\overline M) c_1 (\overline L_1) \cdots c_1 (\overline L_{n-1}) = \lim_i
    \int_{U^{\an}} g_{\overline D} c_1 (\overline M_i) c_1 (\overline L_{1,i}) \cdots c_1 (\overline L_{n-1, i}).
    \label{<+label+>}
  \end{equation}
\end{prop}
\begin{proof}
  By definition of an approximating sequence, there exists a compact subset $\Omega \Subset U^{\an}$
  such that $c_1 (\overline M_i) = 0$ outside $\Omega$.
  Write $\mu_i = c_1 (\overline M_i) c_1 (\overline L_{1,i}) \cdots c_1 (\overline L_{n-1, i})$ and $\mu = c_1 (\overline
  M) c_1 (\overline L_{1}) \cdots c_1 (\overline L_{n-1})$. We also define the $(n-1, n-1)$-current $T_i = c_1 (\overline
  L_{1,i}) \wedge \cdots \wedge c_1 (\overline L_{n-1,i})$ and $T := c_1 (\overline L_1) \wedge \cdots \wedge c_1
  (\overline L_{n-1})$. We have that $\mu_i$ and
  $\mu$ have support in $\Omega$ and $\mu$ is the weak limit of $\mu_i$. Let $q \in \Omega \cap \Supp D$ and let
  $\xi$ be a local equation of $D$ at $p$, then locally at $q$ the function $h = g_{\overline D} + \log \left| \xi
  \right|$ is continuous. Notice that the function $\log \left| \xi \right|$ is smooth and harmonic outside $\supp D$.
  Now let $\Omega ' \Subset \Omega$ be a relatively compact open neighbourhood of $\supp D \cap \Omega$. By Corollaire
  3.3.3 of \cite{chambert-loirFormesDifferentiellesReelles2012} there exists a smooth function $\psi$ such that
  $\psi \equiv 1$ over $\Omega '$ and $\psi \equiv 0$ outside $\Omega$. Using that and the existence of partition of
  unity over $\Omega$ which follows from Proposition 3.3.6 of \cite{chambert-loirFormesDifferentiellesReelles2012} we
  have that there exists a smooth function $\phi : \Omega \setminus \Supp D \rightarrow \R$ which
  is harmonic over $\Omega ' \setminus \Supp D$ and a Green function of $D$ over $\Omega$ (If the absolute value is
  archmidean this follows from classical theory). Write $h = g_{\overline D} +
  \phi$, which is continuous over $\Omega$. Now,
  \begin{equation}
    \lim_i \int_{U^{\an}} g_{\overline D} \mu_i = \lim_i \int_\Omega g_{\overline D} \mu_i = \lim_i \int_\Omega h
    \mu_i - \int_{\Omega} \phi c_1 (\overline M_i) c_1 (\overline L_{1,i}) \cdots c_1 (\overline L_{n-1,i}).
    \label{eq:<+label+>}
  \end{equation}
  The first integral converges towards $\int_{\Omega} h \mu$ by definition of the weak limit of measures and for the
  second term recall that we have by the Poincaré-Lelong formula that $c_1 (\overline M_i) = dd^c g_{\overline M_i}$ and
  $c_1 (\overline M) = dd^c (g_{\overline M})$. Therefore,
  \begin{equation}
    \int_{\Omega} \phi dd^c (g_{\overline M_i}) \wedge T_i = \int_{\Omega} g_{\overline M_i} dd^c \phi \wedge T_i.
    \label{eq:<+label+>}
  \end{equation}
  Notice that $dd^c \phi \wedge T_i$ is supported over the relatively compact subset $\Omega \setminus \Omega '$. Since
  $g_{\overline M_i}$ converges uniformly towards $g_{\overline M}$ over $\Omega$ we have that this integral converges
  towards
  \begin{equation}
    \int_{\Omega} g_{\overline M} dd^c \phi \wedge T = \int_\Omega \phi c_1 (\overline M) \wedge T = \int_\Omega \phi
    \mu.
    \label{eq:<+label+>}
  \end{equation}
  Thus,
  \begin{equation}
    \lim_i \int_{U^{\an}} g_{\overline D} \mu_i = \int_\Omega (h - \phi) \mu = \int_{U^{\an}} g_{\overline
    D} \mu.
    \label{eq:<+label+>}
  \end{equation}
\end{proof}

We now prove the integration by parts formula to finish the proof of Theorem \ref{bigthm:intersection-number}.
\begin{thm}\label{thm:integration-by-part}
  Let $U$ be a quasiprojective variety over $\K_v$ and $\overline M$ be an integrable compactly supported vertical
  line bundle. Let $\overline D$ be an arithmetic divisor and write $D = D_U + D_\infty$ where $D_\infty$ is the
  part supported outside $U$ ($D_\infty$ is a limit of divisors supported outside $U$) and $\overline L_1,
  \dots, \overline L_{n-1}$ be integrable metrised line bundles,
  then
  \begin{align}
    \begin{split}
    \overline M \cdot \overline L_1 \cdots \overline L_{n-1} \cdot \overline D =& \left( \overline M \cdot \overline L_1
    \cdots \overline L_{n-1}\right)_{|D_U} + \lambda c_1 (L_1) \cdots c_1 (L_{n-1}) \cdot D_\infty \\
    & \quad \quad + \int_{U^{\an}} g_{\overline D} c_1 (\overline M) c_1 (\overline L_1) \cdots c_1 (\overline L_{n-1})
    \label{eq:<+label+>}
  \end{split}
  \end{align}
\end{thm}
where  $\lambda$ is the constant value of $g_{\overline M}$ outside any large enough compact subset of
$U^{\an}$.
\begin{proof}
  By Theorem \ref{thm:local-intersection-number} we have
  \begin{equation}
    \overline M \cdot \overline L_1 \cdots \overline L_{n-1} \cdot \overline D = \lim_i \overline M_i \cdot \overline
    L_{1,i} \cdots \overline L_{n-1,i} \cdot \overline D_i.
    \label{eq:<+label+>}
  \end{equation}
  Notice that for every $i, {D_i}_{|U} = D_U$.
  By the integration by parts formula in the projective setting we have
  \begin{align}
    \begin{split}
    \overline M_i \cdot \overline L_{1,i} \cdots \overline L_{n-1,i} \cdot \overline D_i &= \left( \overline M_i \cdot
    \overline L_{1,i} \cdots \overline L_{n-1,i} \right)_{|D_U} + \left( \overline M_i \cdot
    \overline L_{1,i} \cdots \overline L_{n-1,i} \right)_{|{D_i,\infty}} \\
    & \quad \quad + \int_{U^{\an}} g_{\overline D_i} c_1
    (\overline M_i) c_1 (\overline L_{1,i} ) \cdots c_1 (\overline L_{n-1,i}).
    \label{eq:<+label+>}
  \end{split}
  \end{align}
  Now, for any subvariety $Y \subset U$, $\overline M_{|Y}$ is also compactly supported and ${\overline
  M_{i}}_{|Y}$ is also an approximating sequence.
  Furthermore, for any irreducible component $E$ of $D_\infty$ we have that $\overline {M_i}_{|E}$ is the trivial line bundle
  with constant metric equal to $\lambda_i$ where $\lambda_i$ is the constant value of $g_{\overline M_i}$ over any
  large enough compact subset of $U^{\an}$. Thus,
  \begin{align}
    \begin{split}
    \left( \overline M_i \cdot \overline L_{1,i} \cdots \overline L_{n-1,i} \right)_{|D_\infty} &= \lambda_i c_1
    (L_{1,i})_{|D_\infty} \cdots c_1 (L_{n-1,i})_{|D_{i,\infty}} \\
    &= c_1 (L_{1,i}) \cdots c_1 (L_{n-1,i}) \cdot D_{i,\infty} \\
    &\xrightarrow[i \rightarrow  + \infty]{} c_1 (L_1) \cdots c_1 (L_{n-1}) \cdot D_\infty.
    \label{eq:<+label+>}
  \end{split}
  \end{align}
  Finally the result follows from Theorem \ref{thm:local-intersection-number} and
  Proposition \ref{prop:weak-limit-integration} since by the definition of the boundary topology $g_{\overline D_i}$
  converges uniformly over $\Omega \setminus \Supp D_{|U}$ towards $g_{\overline D}$.
\end{proof}

\section{Local arithmetic Hodge index theorem for quasiprojective varieties}
\label{sec:weak-arith-hodge}
We prove the local arithmetic Hodge index theorem from the introduction. If $\overline L,
\overline M$ are integrable adelic line bundles with $\overline L$ nef and $\overline M$ vertical, we say that
$\overline M$ is $\overline L$-bounded if for $\epsilon > 0$ small enough $\overline L \pm \epsilon \overline M$ is nef.

\begin{thm}\label{thm:arithm-hodge-index-local}
  Let $U$ be a quasiprojective variety over a complete field $\K_v$ of dimension $n$. Let
  $\overline M \in \hat \Pic_c (U )$ be an integrable compactly supported vertical line bundle and
$\overline L_1, \dots, \overline L_{n-1}\in \hat \Pic (U )$ be nef metrised line bundles. Then,
  \begin{equation}
    \overline M^2 \cdot \overline L_1 \cdots \overline L_{n-1} \leq 0.
    \label{<+label+>}
  \end{equation}
  Furthermore, if for every $i, L_i^n > 0$ and $\overline M$ is
  $L_i$-bounded. Then we have equality if and only if $g_{\overline M}$ is constant.
\end{thm}
It suffices to prove the result when $\K_v = \C_v$ is complete and algebraically closed by definition of the
intersection number.
  The inequality follows from the arithmetic Hodge index theorem over projective
  varieties (see \cite{yuanArithmeticHodgeIndex2017}) and a limit argument using Theorem
  \ref{thm:local-intersection-number}.
  In particular, we have the following Cauchy-Schwarz inequality: For any compactly supported line bundles
  $\overline M, \overline M '$ and nef adelic line bundles $\overline L_1, \dots, \overline L_{n-1}$, one has
  \begin{equation}
    \left( \overline M \cdot \overline M ' \cdot \overline L_1 \cdots \overline L_{n-1} \right)^2 \leq \left( \overline
    M^2 \cdot \overline L_1 \cdots \overline L_{n-1} \right) \left( \overline M'^2 \cdot \overline L_1 \cdots \overline
    L_{n-1} \right).
    \label{eq:cauchy-schwarz-inequality}
  \end{equation}
  Indeed, this follows from the fact that for any $\lambda \in \R, (\overline M + \overline M')^2 \cdot \overline L_1
  \cdots \overline L_{n-1} \leq 0$.
  We prove the equality case by induction on the dimension of $U$.

  \subsection{Curve case}\label{subsec:curve-case}
  If $U$ is a quasiprojective curve, then
  the $L_{i}$'s do not appear. Let $X$ be the unique projective model of $U$. For any model vertical line bundle
  $\overline N$ over $X$ with metric with compact support in $U^{\an}$ we have
  \begin{equation}
    \overline M \cdot \overline N = 0.
    \label{eq:<+label+>}
  \end{equation}
  By Theorem \ref{thm:local-intersection-number} and Theorem \ref{thm:integration-by-part} we have 
  \begin{equation}
    0 = \overline M \cdot \overline N = \int_{U^{\an}} g_{\overline N} c_1 (\overline M).
    \label{eq:<+label+>}
  \end{equation}
  Since this holds for any model functions $g_{\overline N}$ over $X^{\an}$ with compact support in $U^{\an}$, we get
  the vanishing of the measure $c_1 (\overline M) = 0$ over $U^{\an}$ by Theorem
  \ref{thm:density-model-functions-with-compact-support}.

  Now, let $p,q \in U(\C_v)$ be closed points and let $N = \OO_{X}(p - q)$. There exists a unique
  metrisation $\overline N \in \hat \Pic (X )$ of $N$ such that the metric is flat (see Definition
    5.10 and Theorem 5.11 of \cite{yuanArithmeticHodgeIndex2017}). In particular, $c_1 (\overline N) = 0$ and we have by
    Theorem
  \ref{thm:local-intersection-number}
  \begin{equation}
    \overline M \cdot \overline N = \int_{U^{\an}}g_{\overline M} c_1 (\overline N) = 0.
    \label{eq:<+label+>}
  \end{equation}
  And by Theorem \ref{thm:integration-by-part}, we get
  \begin{equation}
    0 = \overline M \cdot \overline N = g_{\overline M} (p) - g_{\overline M} (q) + \int_{U^{\an}} g_{\overline
    N} c_1 (\overline M) = g_{\overline M} (p) - g_{\overline M} (q).
    \label{eq:<+label+>}
  \end{equation}
  Thus, $g_{\overline M}$ is constant over $U(\C_v) \subset U^{\an}$ which is a dense subset, thus
  $g_{\overline M}$ is constant.

  \subsection{General case}\label{subsec:general-case}
  Suppose $U$ is of dimension $n \geq 2$.
  Let $\overline M_j$ be an approximating sequence of $\overline M$. We
suppose that for every $j \geq 1$ and every $k = 1, \dots, n-1$, the model metrised line bundles $\overline M_j,
\overline L_{k,j}$ are defined over the same model $\sX_j$ which is always possible. Let $\overline
\sD_0$ be a boundary divisor, we can also suppose that $\overline \sD_0$ is defined over every $\sX_j$.
\begin{lemme}\label{lemme:vanishing-of-measure}
  The measure $c_{1}(\overline M)^{2} c_{1}(\overline L_{1}) \cdots c_{1}(\overline L_{n-2})$ is the zero measure.
\end{lemme}
\begin{proof}
  The proof of Lemma 2.5 of \cite{yuanArithmeticHodgeIndex2017} applies in our setting. We show that for any
  compactly supported model vertical metrised line bundle $\overline M'$,
  \begin{equation}
    \label{eq:19}
    \overline M^{2} \overline M' \overline L_{1} \cdots \overline L_{n-2} = 0.
  \end{equation}
  In particular by Theorem \ref{thm:integration-by-part}, we have
  \begin{equation}
    \label{eq:20}
    \int_{U^{\an}} g_{\overline M'} c_{1}(\overline M)^{2} c_{1}(\overline L_{1}) \cdots c_{1}(\overline L_{n-2}) = 0.
  \end{equation}
  And by Theorem \ref{thm:density-model-functions-with-compact-support}, we have the vanishing of the measure.

  Thus, it remains to show
  \eqref{eq:19}. By the Cauchy-Schwarz inequality, we have
  \begin{equation}
    \overline M \cdot \overline M ' \cdot \overline L_1 \cdots \overline L_{n-1} = 0.
    \label{eq:<+label+>}
  \end{equation}
  Thus, \eqref{eq:19} is equivalent to
  \begin{equation}
    \overline M \cdot \overline M' \cdot \overline L_1 \cdots \overline L_{n-2} \cdot (\overline L_{n-1} \pm \epsilon
    \overline M) =0
    \label{eq:inequality-to-show}
  \end{equation}
  for $\epsilon > 0$ small enough. Now, for $\epsilon > 0$ small enough $\overline L_{n-1} \pm \epsilon \overline M$
  is nef. Write $\overline L_{\pm \epsilon} = \overline L_{n-1} \pm \epsilon \overline M$. We apply the Cauchy-Schwarz
  inequality again

  \begin{equation}
    \left( \overline M \cdot \overline M ' \cdot \overline L_1 \cdots \overline L_{n-2} \overline L_{\pm \epsilon}
    \right)^2 \leq \left( \overline M^2 \overline L_1 \cdots \overline L_{n-2} \cdot \overline L_{\pm
    \epsilon} \right) \left( \overline M'^2 \cdot \overline L_1 \cdots \overline L_{n-2} \cdot \overline L_{\pm \epsilon}
    \right).
    \label{eq:<+label+>}
  \end{equation}
  Therefore, it suffices to show that $\overline M^2 \overline L_1 \cdots \overline L_{\pm \epsilon} = 0$. We have
  \begin{align}
    0 = \overline M^2 \cdot \overline L_1 \cdots \overline L_{n-1} &= \overline M^2 \cdot \overline L_1 \cdots \overline
    L_{n-2} \left( \frac{1}{2} \overline L_{n-1} + \frac{\epsilon}{2} \overline M + \frac{1}{2} \overline L_{n-1} -
      \frac{\epsilon}{2} \overline M \right) \\
    & = \frac{1}{2} \overline M^2 \cdot \overline L_1 \cdots \overline L_{\epsilon} + \frac{1}{2} \overline M^2 \cdot
    \overline L_1  \cdots \overline L_{- \epsilon} \leq 0
    \label{eq:<+label+>}
  \end{align}
  where the last inequality comes from the inequality of the arithmetic Hodge index theorem. Thus, we get $\overline M^2
  \overline L_1 \cdots \overline L_{\pm \epsilon} = 0$.
\end{proof}

  \begin{lemme}\label{lemme:diviseur-plus-grand-que-ample}
  Let $\overline L \in \hat \Pic (U )$ be a nef metrised line bundle such that $L^{n} > 0$. Write $\overline L =
  \OO_U (\overline D)$ and let $\overline D_{k}$ be a Cauchy sequence defining $\overline
  D$. For every codimension 1 closed subvariety $Y \subset U \setminus \Supp D_{|U}$, there exists $k_0 \geq 0$
  and a rational number $c > 0$ such that for every $k \geq k_0$, there exists an integer $n_k$ such that $n_k L_{k}$
  admits a global section $s_k$ satisfying $\frac{1}{n_k} \ord_Y (s_k) \geq c$.
\end{lemme}
\begin{proof}
  Let $X$ be a completion of $U$ and $D_0$ be a boundary divisor on $X$ and let $H$ be an ample effective divisor over
  $X$ such that $Y \subset \Supp H$. We can suppose that every $X_i$ where $D_i$ is defined admits a birational morphism
  $X_i \rightarrow X$. Because $L$ is nef, we have $L^{n-1} \cdot H \geq 0$ and therefore there exists a rational number
  $s > 0$ such that $L^n - s L^{n-1} \cdot H > 0.$ Since $L_k^n \rightarrow L^n$ and $L^{n-1}_k \cdot H \rightarrow
  L^{n-1} \cdot H$, there exists an integer $k_0 \geq 0$ such that for every $k \geq k_0$, $L_k^n - s L_k^{n-1} \cdot H
  > 0$. In particular, this implies by Theorem \ref{thm:inequality-volume}
  \begin{equation}
    \forall k \geq k_0, \quad \vol (X_k, L_k - sH) \geq L_k^{n} - s L_k^{n-1} \cdot H >0.
    \label{eq:<+label+>}
  \end{equation}
  Therefore, there exists an integer $n_k$ such that $\Gamma (X_k, n_k (D_k - sH)) \neq 0$. Since $Y \not \subset \supp D_{|U}$, this
  yields a global section $s_k$ of $n_k L_k$ such that
  \begin{equation}
    \frac{1}{n_k} \ord_Y (s_k) \geq s \ord_Y (H) =: c.
    \label{eq:<+label+>}
  \end{equation}
\end{proof}

Now, we write $\overline L_{n-1} = \OO_U (\overline D_{n-1})$ and we replace $U$ by $U \setminus \supp
{D_{n-1}}_{|U}$. If $g_{\overline M}$ is constant over $(U \setminus \supp {D_{n-1}}_{|U})^{\an}$ then it will be
constant over $U^{\an}$ by density.

\begin{prop}\label{prop:induction}
  For any codimension 1 closed subvariety $Y \subset U$, we have
  \begin{equation}
    \overline M_{|Y}^2 \cdot \overline{L_1}_{|Y} \cdots \overline{L_{n-2}}_{|Y} = 0.
    \label{eq:<+label+>}
  \end{equation}
\end{prop}
\begin{proof}
Let $Y \subset U$ be a codimension 1 subvariety. Recall that we have assumed that for every $1 \leq k \leq n-1$ and for
every $i \geq 0$, $L_{k,i}$ and $M_i$ are defined over the same model $X_i$. By Lemma \ref{lemme:diviseur-plus-grand-que-ample},
there exists $k_0 \geq 0$ and a number $c >0$ such that for every $k \geq k_0$, $n_k L_{n-1, k}$ admits a global section
$s_k$ over $X_k$ such that
\begin{equation}
  \frac{1}{n_k} \ord_{Y_k} (s_k) \geq c
  \label{eq:<+label+>}
\end{equation}
for some integer $n_k \geq 1$ where $Y_k$ is the closure of $Y$ in $X_k$.
Therefore, over $X_k$ we have
\begin{equation}
  \frac{1}{n_k} \div (s_k) = cY_k + D_k
  \label{eq:<+label+>}
\end{equation}
where $D_k$ is some effective divisor. Now we have
\begin{equation}
  \overline M^2 \cdot \overline L_1 \cdots \overline L_{n-1} = \lim_k \overline M^2 \cdot \overline L_1 \cdots \overline
  L_{n-2} \cdot \overline L_{n-1,k}.
  \label{eq:<+label+>}
\end{equation}
Indeed, by Theorem \ref{thm:local-intersection-number}, we have $\overline M^2 \cdot \overline L_1 \cdots \overline
L_{n-1} = \int_\Omega g_{\overline M} c_1(\overline M) \cdot c_1 (\overline L_1) \cdots c_1 (\overline L_{n-1})$ and
$c_1 (\overline L_{n-1})$ is the limit of $c_1 (\overline L_{n-1,k})$ over the compact subset $\Omega$.
Using that $\overline L_{n-1,k} = \OO_{X_k} \left(\frac{1}{n_k} \hat \div(s_k) \right)$ we have by Theorem
\ref{thm:integration-by-part}
\begin{align}
  \begin{split}
  \overline M^2 \cdot \overline L_1 \cdots \overline L_{n-2} \cdot \overline L_{n-1,k} &= c \left( \overline M^2 \cdot
  \overline L_1 \cdots \overline L_{n-2}  \right)_{|Y} + \left( \overline M^2 \cdot
  \overline L_1 \cdots \overline L_{n-2}  \right)_{| {D_k}_{|U}} \\
  & \quad \quad - \int_{U^{\an}} \log \left| s_k \right| c_1 (\overline M^2) c_1 (\overline L_1) \cdots c_1 (\overline L_{n-2}).
  \label{eq:<+label+>}
\end{split}
\end{align}
Indeed, since $M$ is a limit of trivial line bundles $M_i$ the geometric intersection number in the integration by parts formula vanishes.
The integral vanishes by Lemma \ref{lemme:vanishing-of-measure} and by the inequality of the arithmetic Hodge index theorem
\begin{equation}
  \left( \overline M^2 \cdot \overline L_1 \cdots \overline L_{n-2}  \right)_{| {D_k}_{|U}} \leq 0.
  \label{eq:<+label+>}
\end{equation}
Therefore letting $k \rightarrow + \infty$ we get
\begin{equation}
  0 = \overline M^2 \cdot \overline L_1 \cdots \overline L_{n-1} \leq c \left( \overline M^2 \cdot
  \overline L_1 \cdots \overline L_{n-2}  \right)_{|Y} \leq 0.
  \label{eq:<+label+>}
\end{equation}
and the proposition is shown.
  \end{proof}
    \begin{proof}[Proof of Theorem \ref{thm:arithm-hodge-index-local}]
      To conclude the proof, let $Y$ be a codimension 1 closed subvariety of $U$. We have by Proposition
      \ref{prop:induction} that
      \begin{equation}
        \left( \overline M_{|Y}^2 \cdot \overline{L_1}_{|Y} \cdots \overline{L_{n-2}}_{|Y} = 0 \right).
        \label{eq:<+label+>}
      \end{equation}
      We claim that apart from a finite number of codimension 1 subvarieties of $U$, we have ${L_i}_{|Y}^{n-1} > 0$.
      Indeed, write $\overline L_i = \OO_U (\overline D_i)$ and let $X$ be a completion where $D_0$ is defined and
      $H$ be an ample effective divisor over $X$. We assume that for every $k$, $D_{i,k}$ is defined over the same model
      $X_k$ which admits a birational morphism $X_k \rightarrow X$. We can also assume
      that the $\epsilon_k$'s defining the Cauchy sequences for every $\overline D_i$ are the same. We
      argue as in the proof of Lemma \ref{lemme:diviseur-plus-grand-que-ample} to show that there exists $k_0$ and $s >
      0$ such that for every $i = 1, \dots, n -2$ and for every $k \geq k_0, \vol (D_{i, k} - s H) > 0$.  Let $k_0$ be
      large enough such that for every $k \geq k_0$, we have $s H - \epsilon_k D_0$ is ample over $X$. Now since $\vol (D_{i,k_0}
      - s H) > 0$, there exists an integer $M_i$ and a rational function $\phi_i$ over $X_{k_0}$ such that
      \begin{equation}
        D_{i,k_0} - s H + \frac{1}{M_i} \div_{X_{k_0}} (\phi_i) \geq 0.
        \label{eq:<+label+>}
      \end{equation}
      Now, for every $k \geq k_0$, we have over $X_k$
      \begin{equation}
        D_{i,k} - s H + \frac{1}{M_i} \div_{X_k} (\phi_i) \geq - \epsilon_{k_0} D_0.
        \label{eq:<+label+>}
      \end{equation}
      Let $V \subset U$ be the closed subvariety defined by
      \begin{equation}
        V = \left( \bigcup_i \supp {D_i}_{|U} \right) \cup \left( \bigcup_i \supp {\div (\phi_i)}_{|U} \right) \cup
        \Supp H_{|U}.
        \label{eq:<+label+>}
      \end{equation}
      Let $Y \subset U$ be a codimension 1 subvariety such that $Y \not \subset V$ and write
      $\overline Y_k$ for the closure of $Y$ in $X_k$. We also write $\overline Y$ for the closure of $Y$ in $X$. Then,
      we also have
      \begin{equation}
        {D_{i,k}}_{|\overline Y_k} + \frac{1}{M_i} \div (\phi_{i |\overline Y_k}) \geq s H_{|\overline Y_k} -
        \epsilon_{k_0} {D_0}_{|\overline Y_k}.
        \label{eq:<+label+>}
      \end{equation}
      In particular, we get for every $k \geq k_0$ (because $L_{i,k}$ and $sH - \epsilon_{k_0}D_0$ are nef over
      $X_k$)
      \begin{equation}
        {L_{i}}_{|Y}^{n-1} = \lim_k {L_{i,k}}_{|\overline Y_k}^{n-1} \geq \lim_k (s H_{|\overline Y_k} -
          \epsilon_{k_0} {D_0}_{|\overline Y_k})^{n-1}.
        \label{eq:<+label+>}
      \end{equation}
      And we can conclude because $sH - \epsilon_{k_0} D_0$ is ample over $X$ and
      \begin{equation}
        (s H_{|\overline Y_k} - \epsilon_{k_0} {D_0}_{|\overline Y_k})^{n-1} = (s H - \epsilon_{k_0} D_0)^{n-1} \cdot
        \overline Y_k = (s H - \epsilon_{k_0} D_0)^{n-1} \cdot \overline Y > 0
        \label{eq:<+label+>}
      \end{equation}
      where $\overline Y$ is the closure of $Y$ in $X$.
      By induction ${g_{\overline M}}_{|Y}$ is constant. By Bertini's theorem, if $x,y$ are two closed point of
      $U \setminus V$ then there exists an irreducible hyperplane section $Y$ of $X$ that contains $x,y$. We conclude that
      $g_{\overline M}(x) = g_{\overline M} (y)$. By
      density of the closed point in the Berkovich analytification we get that $g_{\overline M}$ is constant over
      $(U \setminus V)^{\an}$ and thus over $U^{\an}$.
  \end{proof}

  We can now prove Theorem \ref{bigthm:calabi}.
\begin{cor}[Calabi theorem]\label{CorCalabiThm}
  Let $\overline L_{1}, \overline L_{2}$ be two nef metrised line bundles over a
  quasiprojective variety $U$ over $\C_v$ such that $c_1 (\overline L_1)^n = c_1 (\overline L_2)^n$ and
  $\overline M = \overline L_{1} - \overline L_{2}$ is a compactly supported vertical line bundle (in particular $L_1 = L_2 =: L$). If $L^n >0$, then $g_{\overline M}$ is constant.
\end{cor}
\begin{proof}
  Let $\overline M = \overline L_{1} - \overline L_{2}$ and let $g_{\overline M}$ be its associated Green function. By
  Theorem \ref{thm:local-intersection-number}, we have
  \begin{equation}
    \label{eq:18}
    \overline M \cdot \overline L_{1}^{n} = \int_{X^{\an}} g_{\overline M} c_{1}(\overline L_{1})^{n} = \int_{X^{\an}}
    g_{\overline M} c_{1}(\overline L_{2})^n = \overline M \cdot \overline L_{2}^{n}.
  \end{equation}
  Thus, we get
  \begin{equation}
    \label{eq:21}
    0 = \overline M \cdot \left( \overline L_1^n - \overline L_2^n \right) = \overline M^{2} \cdot \left(\sum_{k=0}^{n-1}
    \overline L_{1}^{k} \overline L_{2}^{n-1-k}\right).
  \end{equation}
  By the inequality of the arithmetic Hodge index theorem every term in that sum is 0. Thus, we have
  \begin{equation}
    \label{eq:22}
    \overline M^{2}\cdot (\overline L_{1} + \overline L_{2})^{n-1} = 0.
  \end{equation}
  Now, it is clear that $\overline M$ is $\overline L_{1} + \overline L_{2}$-bounded and that $(L_1 + L_2)^n >0$, by the
  equality case of the arithmetic Hodge index theorem we have that $g_{\overline M}$ is constant.
\end{proof}

\bibliographystyle{alpha}
\bibliography{biblio}

\end{document}